\documentclass[12pt,letterpaper,reqno]{amsart}

\usepackage{times}
\usepackage[T1]{fontenc}
\usepackage{mathrsfs}
\usepackage{latexsym}
\usepackage[dvips]{graphics}
\usepackage{epsfig}
\usepackage{amsmath,amsfonts,amsthm,amssymb,amscd}
\input amssym.def
\input amssym.tex
\usepackage{color}
\usepackage{comment}
\usepackage{color}
\usepackage{hyperref}
\usepackage{url}

\newcommand{\burl}[1]{\textcolor{blue}{\url{#1}}}

\newcommand\be{\begin{equation}}
\newcommand\ee{\end{equation}}
\newcommand\nbea{\begin{eqnarray*}}
\newcommand\neea{\end{eqnarray*}}

\newcommand\bea{\begin{eqnarray}}
\newcommand\eea{\end{eqnarray}}
\newcommand\bi{\begin{itemize}}
\newcommand\ei{\end{itemize}}
\newcommand\ben{\begin{enumerate}}
\newcommand\een{\end{enumerate}}
\newcommand\bc{\begin{center}}
\newcommand\ec{\end{center}}
\newcommand\ba{\begin{array}}
\newcommand\ea{\end{array}}

\newtheorem{thm}{Theorem}[section]

\newtheorem{lem}[thm]{Lemma}

\newcommand{\twocase}[5]{#1 \begin{cases} #2 & \text{{\rm #3}}\\ #4
&\text{{\rm #5}} \end{cases}   }

\newcommand{\hr}[1]{\href{#1}{\url{#1}}}

\numberwithin{equation}{section}

\title[Gaussian Behavior in Zeckendorf Decompositions in Small Intervals]{Gaussian Behavior of the Number of Summands in Zeckendorf Decompositions in Small Intervals}

\author[Best]{Andrew Best}
\email{\textcolor{blue}{\href{mailto:ajb5@williams.edu)
}{ajb5@williams.edu}}}
\address{Department of Mathematics and Statistics, Williams College, Williamstown, MA 01267}

\author[Dynes]{Patrick Dynes}
\email{\textcolor{blue}{\href{mailto:pdynes@clemson.edu)
}{pdynes@clemson.edu}}}
\address{Department of Mathematical Sciences, Clemson University, Clemson, SC 29634}

\author[edelsbrunner]{Xixi Edelsbrunner}
\email{\textcolor{blue}{\href{mailto:xe1@williams.edu}{xe1@williams.edu}}}
\address{Department of Mathematics and Statistics, Williams College, Williamstown, MA 01267}

\author[McDonald]{Brian McDonald}
\email{\textcolor{blue}{\href{mailto:bmcdon11@u.rochester.edu}{bmcdon11@u.rochester.edu}}}
\address{Department of Mathematics, University of Rochester, Rochester, NY 14627}

\author[Miller]{Steven J. Miller}
\email{\textcolor{blue}{\href{mailto:sjm1@williams.edu}{sjm1@williams.edu}},  \textcolor{blue}{\href{Steven.Miller.MC.96@aya.yale.edu}{Steven.Miller.MC.96@aya.yale.edu}}}
\address{Department of Mathematics and Statistics, Williams College, Williamstown, MA 01267}

\author[Tor]{Kimsy Tor}
\email{\textcolor{blue}{\href{mailto:ktor.student@manhattan.edu}{ktor.student@manhattan.edu}}}
\address{Department of Mathematics, Manhattan College, Riverdale, NY 10471}

\author[Turnage-Butterbaugh]{Caroline Turnage-Butterbaugh}
\email{\textcolor{blue}{\href{mailto:cturnagebutterbaugh@gmail.com}{cturnagebutterbaugh@gmail.com}}}
\address{Department of Mathematics, North Dakota State University, Fargo, ND 58102}

\author[Weinstein]{Madeleine Weinstein}
\email{\textcolor{blue}{\href{mailto:mweinstein@g.hmc.edu}{mweinstein@g.hmc.edu}}}
\address{Department of Mathematics, Harvey Mudd College, Claremont, CA 91711 }

\thanks{This research was conducted as part of the 2014 SMALL REU program at Williams College and was supported by NSF grant DMS1347804 and DMS1265673,  Williams College, and the Clare Boothe Luce Program of the Henry Luce Foundation. It is a pleasure to thank the participants of the SMALL REU and the 16\textsuperscript{th} International Conference on Fibonacci Numbers and their Applications for helpful discussions.}



\date{\today}

\begin{document}

\begin{abstract}
Zeckendorf's theorem states that every positive integer can be written uniquely as a sum of non-consecutive Fibonacci numbers ${F_n}$, with initial terms $F_1 = 1, F_2 = 2$. We consider the distribution of the number of summands involved in such decompositions. Previous work proved that as $n \to \infty$ the distribution of the number of summands in the Zeckendorf decompositions of $m \in [F_n, F_{n+1})$, appropriately normalized, converges to the standard normal. The proofs crucially used the fact that all integers in $[F_n, F_{n+1})$ share the same potential summands.

We generalize these results to subintervals of $[F_n, F_{n+1})$ as $n \to \infty$; the analysis is significantly more involved here as different integers have different sets of potential summands. Explicitly, fix an integer sequence $\alpha(n) \to \infty$. As $n \to \infty$, for almost all $m \in [F_n, F_{n+1})$ the distribution of the number of summands in the Zeckendorf decompositions of integers in the subintervals $[m, m + F_{\alpha(n)})$, appropriately normalized, converges to the standard normal. The proof follows by showing that, with probability tending to $1$, $m$ has at least one appropriately located large gap between indices in its decomposition. We then use a correspondence between this interval and $[0, F_{\alpha(n)})$ to obtain the result, since the summands are known to have Gaussian behavior in the latter interval. 
\end{abstract}

\maketitle

\tableofcontents


\section{Introduction}

\subsection{History}

Let $\{F_n\}$ denote the Fibonacci numbers, normalized so that $F_1 = 1$, $F_2 = 2$\footnote{We define the sequence this way to retain uniqueness in our decompositions}, and $F_{n+1} = F_n + F_{n-1}$. An interesting equivalent definition of the Fibonacci numbers is that they are the unique sequence of positive integers such that every positive integer has a unique legal decomposition as a sum of non-adjacent terms. This equivalence is known as Zeckendorf's theorem \cite{Ze} and has been extended by many authors to a variety of other sequences.

For the Fibonacci numbers, Lekkerkerker \cite{Lek} proved that the average number of summands needed in the Zeckendorf decomposition of an integer $m \in [F_n, F_{n+1})$ is $\frac{n}{\varphi^2+1} + O(1)$, where $\varphi = \frac{1+\sqrt{5}}{2}$, the golden mean, is the largest root of the Fibonacci recurrence. This has been extended to other positive linear recurrence sequences, and much more is known. Namely, the distribution of the number of summands converges to a Gaussian as $n\to\infty$. There are several different methods of proof, from continued fractions to combinatorial perspectives to Markov processes. See \cite{Day,DDKMV,DG,FGNPT,GT,GTNP,Ke,KKMW,LT,Len,MW1,MW2,Ste1,Ste2} for a sampling of results and methods along these lines, \cite{Al,CHMN1,CHMN2,CHMN3,DDKMMV,DDKMV} for generalizations to other types of representations, and \cite{BBGILMT,BILMT} for related questions on the distribution of gaps between summands.

The analysis in much of the previous work was carried out for $m \in [F_n, F_{n+1})$. The advantage of such a localization\footnote{As the sequence $\{F_n\}$ is exponentially growing,  it is easy to pass from $m$ in this interval to $m \in [0, F_n)$.} is that each $m$ has the same candidate set of summands and is of roughly the same size. The purpose of this work is to explore some of the above questions on a significantly smaller scale and determine when and how often we obtain Gaussian behavior. Note that we cannot expect such behavior to hold for all sub-intervals of $[F_n, F_{n+1})$, even if we require the size to grow with $n$. To see this, consider the interval \be [F_{2n} + F_{n} + F_{n-2} + \cdots + F_{\lfloor n^{1/4} \rfloor}, F_{2n} + F_{n+1} + F_{\lfloor n^{1/4}\rfloor}). \ee The integers in the above interval that are less than $F_{2n} + F_{n+1}$ have on the order of $n/2$ summands, while those that are larger have at most on the order of $n^{1/4}$ summands. Thus the behavior cannot be Gaussian.\footnote{Though in this situation it would be interesting to investigate separately the behavior on both sides.}

\subsection{Main Result}\hspace*{\fill} \\

Fix any increasing positive integer valued function $\alpha(n)$ with \be\label{eq:condsonalpha} \lim_{n\to\infty}{\alpha(n)} \ = \ \lim_{n\to\infty}\left(n-\alpha(n)\right)\ = \ \infty.\ee  Our main result, given in the following theorem, extends the Gaussian behavior of the number of summands in Zeckendorf decompositions to smaller intervals. Note that requiring $m$ to be in $[F_n, F_{n+1})$ is not a significant restriction because given any $m$, there is always an $n$ such that this holds.

\begin{thm}[Gaussianity on small intervals] \label{thm:mainthm} For $\alpha(n)$ satisfying \eqref{eq:condsonalpha}, the distribution of the number of summands in the decompositions of integers in the interval $[m,m+F_{\alpha(n)})$ converges to a Gaussian distribution when appropriately normalized for almost all $m\in [F_n,F_{n+1})$. Specifically, using the notation from equations \eqref{eq:decompmpreCs} and \eqref{eq:decompmintoC1C2C3}, the Gaussian behavior holds for all $m$ where there is a gap of length at least 2 in the $C_2(m)$ (and $q(n) = o(\sqrt{n})$ is an increasing even function that diverges to infinity).
\end{thm}


\section{Preliminaries}\label{sec:preliminaries}

In order to prove Theorem \ref{thm:mainthm}, we establish a correspondence between the decompositions of integers in the interval $[m,m+F_{\alpha(n)})$ and those in $[0,F_{\alpha(n)})$.  We first introduce some notation.  Fix a non-decreasing positive function $q(n)<n-\alpha(n)$ taking on even integer values with the restriction that $q(n)\to\infty$; later (see \eqref{eq:restrictiononq}) we will see that we must also take $q(n) = o(\sqrt{n})$. For $m\in [F_n,F_{n+1})$ with decomposition
\begin{align}\label{eq:decompmpreCs}
m & \ = \ \sum_{j=1}^n{a_jF_j},
\end{align}
define
\begin{align}\label{eq:decompmintoC1C2C3}
& C_1(m) \ := \ (a_1,a_2,...,a_{\alpha(n)}), \nonumber\\
& C_2(m) \ := \ (a_{\alpha(n)+1},...,a_{\alpha(n)+q(n)}), \ \text{and} \nonumber\\
& C_3(m) \ := \ (a_{\alpha(n)+q(n)+1},...,a_n).
\end{align}
Note that each $a_i \in \{0,1\}$ for all $1 \le i \le n$. Let $s(m)$ be the number of summands in the decomposition of $m$. That is, let
\begin{align}
s(m) \ := \ \sum_{j=1}^n{a_j}.
\end{align}
 Similarly, let $s_1(m), \ s_2(m)$, and $s_3(m)$ be the number of summands contributed by $C_1(m), \ C_2(m)$, and $C_3(m)$ respectively. Note that no two consecutive $a_j$'s equal $1$.

\begin{lem}
Let $x\in[m,m+F_{\alpha(n)})$. If there are at least two consecutive 0's in $C_2(m)$, then $C_3(x)$ is constant, and hence $s_3(x)$ is constant as well.
\end{lem}

\begin{proof}
Assume there are at least two consecutive 0's in $C_2(m)$. Then for some $k\in$ $[\alpha(n)+2$, $\alpha(n)+q(n))$, we have $a_{k-1}=a_k=0$.  Let $m'$ denote the integer obtained by truncating the decomposition of $m$ at $a_{k-2}F_{k-2}$. (Note that if $a_{k-2}=1$, we include $F_{k-2}$ in the truncated decomposition, and if $a_{k-2}=0$ we do not.)  Then $m'<F_{k-1}$. Since $F_{\alpha(n)}\leq F_{k-2}$, it follows that for any $h<F_{\alpha(n)}$ we have
\begin{align}
m'+h  \ < \ F_{k-1}+F_{k-2}\ =\ F_{k},
\end{align}
and thus the decomposition of $m'+h$ has largest summand no greater than $F_{k-1}$.  Therefore, the Zeckendorf decomposition of $m+h$ is obtained simply by concatenating the decompositions for $m-m'$ and $m'+h$.  Hence $C_3(m+h)=C_3(m-m')=C_3(m)$.
\end{proof}

With this lemma, we see that the distribution of the number of summands involved in the decomposition of $x\in [m,m+F_{\alpha(n)})$ depends (up to a shift) only on what happens in $C_1(x)$ and $C_2(x)$, \textbf{\emph{provided}} that there is a gap between summands of length at least two somewhere in $C_2(m)$.  In light of this stipulation, we will show the following items in order to prove our main theorem.

\begin{itemize}

\item With high probability, $m$ is of the desired form (i.e., there is a gap between summands of length at least two in $C_2(m)$). \vspace{5pt}

\item When $m$ is of the desired form, the distribution of the number of summands involved in $C_1(x)$ for $x\in [m,m+F_{\alpha(n)})$ converges to Gaussian when appropriately normalized. \vspace{5pt}

\item The summands involved in $C_2(x)$ produce a negligible error term (i.e., there are significantly fewer summands from $C_2(x)$ than there are from $C_1(m)$).

\end{itemize}

We address the first point with the following lemma.

\begin{lem} \label{lem:2consecutive0s}
With probability $1+o(1)$, there are at least 2 consecutive 0's in $C_2(m)$ if $m$ is chosen uniformly at random from the integers in $[F_n, F_{n+1})$.
\end{lem}

\begin{proof}
Suppose $m$ is not of the desired form. Recalling that $q(n)$ takes on even integer values, it follows that either $C_2(m) = (1,0,1,0,\dots,1,0)$ or $C_2(m) = (0,1,0,1,\dots,0,1)$. For each of these two cases, we now count the total number of ways to choose the coefficients for $C_3(m)$ and $C_1(m)$.

In the former case, we have $a_{\alpha(n)+q(n)}=0$. Thus the number of ways to choose the coefficients for $C_3(m)$ is equal to the number of ways to legally construct
\begin{equation}\label{eqn:numberofways}
\sum_{j=\alpha(n)+q(n)+1}^{n}{a_jF_j}
\end{equation}
with no nonzero consecutive coefficients and $a_n = 1$ (since $m \in [F_n, F_{n+1})$ we must select $F_n$). There are $F_{n-\alpha(n)-q(n)-1}$ ways to make such a construction, so we conclude that the number of ways to choose the coefficients for $C_3(m)$ is equal to $F_{n-\alpha(n)-q(n)-1}$. To see this, we argue as in \cite{BBGILMT,BILMT}. By shifting indices, the number of legal constructions here is the same as the number of legal  ways to choose the coefficients in\be \sum_{j=1}^{n-\alpha(n)-q(n)} \widetilde{a}_j F_j \ee where we must choose the final summand. By Zeckendorf's theorem, this is equivalent to counting the number of elements in $[F_{n-\alpha(n)-q(n)}, F_{n-\alpha(n)-q(n)+1})$, which by the Fibonacci recurrence is just $F_{n-\alpha(n)-q(n)-1}$.  Thus the number of ways to choose the coefficients for $C_3(m)$ is equal to $F_{n-\alpha(n)-q(n)-1}$. Similarly, since $a_{\alpha(n)}=0$ the number of ways to choose the coefficients for $C_1(m)$ is equal to $F_{\alpha(n)}$. Thus, if $C_2(m) = (1,0,1,0,\dots,1,0)$, there are $F_{n-3-\alpha(n)-q(n)}F_{\alpha(n)}$ ways to choose the coefficients for $C_3(m)$ and $C_1(m)$.

A similar counting argument shows that if $C_2(m) = (0,1,0,1,\dots,0,1)$, then the coefficients for $C_3(m)$ and $C_1(m)$ can be chosen in $F_{n-\alpha(n)-q(n)-2}F_{\alpha(n)+1}$ different ways. Therefore, since $q(n)\to\infty$ as $n\to \infty$, the probability of $m$ not being of the desired form is \begin{align}
\frac{F_{n-\alpha(n)-q(n)-1}F_{\alpha(n)}+F_{n-\alpha(n)-q(n)-2}F_{\alpha(n)+1}}{F_{n-1}}\ \sim\ \frac{2}{\sqrt{5}}\phi^{-q(n)}\ =\ o(1).
\end{align}
\end{proof}

Assuming $m$ is of the desired form, we now consider the distribution of $s(x)$ for $x\in [m,m+F_{\alpha(n)})$.

\begin{lem}\label{error}
If $m$ has at least 2 consecutive 0's in $C_2(m)$, then for all $x\in [m,m+F_{\alpha(n)})$, we have
\begin{align}
0\ \leq\ s(x)-s_3(m)-s(t(x)) \ < \ q(n),
\end{align}
where $t(x)$ denotes some bijection
\begin{align}
t:\mathbb{Z}\cap [m,m+F_{\alpha(n)})\to\mathbb{Z}\cap [0,F_{\alpha(n)}).
\end{align}
\end{lem}

\begin{proof}
First, note that the number of summands in the decomposition of $x$ with indices $i\in$ $[\alpha(n)$, $\alpha(n)+q(n))$ must be less than $q(n)$. Next, let $m_0$ be the sum of the terms in the decomposition of $x$ truncated at $a_{\alpha(n)-1}F_{\alpha(n)-1}$.  Define the bijection $t$ by
\begin{align}
\twocase{t(m+h)\ := \ }{m_0+h,}{if $m_0+h<F_{\alpha(n)}$}{m_0+h-F_{\alpha(n)}}{if $m_0+h\geq F_{\alpha(n)}$.}
\end{align}
For any $x\in [m,m+F_{\alpha(n)})$, the decompositions of $t(x)$ and $x$ agree for the terms with index less than $\alpha(n)$.  Furthermore, the decompositions of $x$ and $m$ agree for terms with index greater than $\alpha(n)+q(n)$.  Therefore, the number of summands in the decomposition of $x$ with indices $i\in [\alpha(n),\alpha(n)+q(n))$ is equal to $s(x)-s_3(m)-s(t(x))$. Combining this with our initial observation, the lemma now follows.
\end{proof}

As a result of this lemma, the distribution of $s(x)$ over the integers in $[m,m+F_{\alpha})$ is a shift of its distribution over $[0,F_{\alpha(n)})$, up to an error bounded by $q(n)$.  With this fact, we are now ready to prove the main theorem.

\section{Proof of Theorem \ref{thm:mainthm}}\label{sec:proofofthmmain}
We now prove our main result. The key idea is that with probability approaching $1$, we have a gap of length at least $2$ in the middle summands of our decompositions, and this allows us to use our bijection to reduce  questions on the distribution of the number of summands in $[m, m + F_{\alpha(n)})$ to similar statements on $[0, F_{\alpha(n)})$. In doing so, the fluctuations in the difference between the two quantities is bounded by $q(n)$, which is a free parameter in our splitting of the decomposition, and can therefore be taken to be sufficiently small.
\begin{proof}
For a fixed $m\in [F_n,F_{n+1})$ with two consecutive 0's somewhere in $C_2(m)$, we define random variables $X_n$ and $Y_n$ by
\begin{align}
& X_n \ := \ s(X), \ \ \
Y_n \ := \ s(Y),
\end{align}
where $X$ is chosen uniformly at random from $\mathbb{Z}\cap[m,m+F_{\alpha(n)})$ and $Y$ is chosen uniformly at random from $\mathbb{Z}\cap [0,F_{\alpha(n)})$. Let
\begin{align}
& X_n' \ := \ \frac{1}{\sigma_x(n)}(X_n-\mathbb{E}[X_n]),
\intertext{and}
& Y_n' \ := \ \frac{1}{\sigma_y(n)}(Y_n-\mathbb{E}[Y_n]),
\end{align}
where $\sigma_x(n)$ and $\sigma_y(n)$ are the standard deviations of $X_n$ and $Y_n$, respectively, so that $X_n$ and $Y_n$ are normalized with mean 0 and variance 1. It is known that the densities of $Y_n'$ converge to the density of the standard normal\footnote{Many of the references give proofs both for the case of the Fibonacci numbers as well as for more general recurrences; see \cite{KKMW} for a simple proof using just Stirling's formula, which yields that the mean grows on the order of $\alpha(n)$ and the standard deviation grows on the order of $\sqrt{\alpha(n)}$.}, and we claim that $X_n'$ converges to the standard normal as well. Though we only need the order of magnitude of $\sigma_y(n)$, for completeness we remark that the mean of $Y_n$ is $\frac{n}{\varphi+2} + O(1)$ and the variance $\sigma_y(n)^2$ is $\frac{\varphi n}{5(\varphi+2)} + O(1)$, where $\varphi = \frac{1+\sqrt{5}}{2}$ is the golden mean.

Let $f_n$ and $g_n$ be the cumulative density functions for $X_n'$ and $Y_n'$, respectively.  By Lemma \ref{error}, we have
\begin{align}
g_n\left(x-\frac{q(n)}{\sigma_y(n)}\right) \ \leq \ f_n(x) \ \leq \ g_n\left(x+\frac{q(n)}{\sigma_y(n)}\right).
\end{align}
Since $\sigma_y(n)\to\infty$, we may add the restriction to $q(n)$ that \be\label{eq:restrictiononq} q(n)\ = \ o\left(\sigma_y(n)\right) \ = \ o\left(\sqrt{n}\right).\ee  Since $\{g_n\}_n$ converges pointwise to the cumulative distribution function for the standard normal, say $g(x)$, and since
\begin{align}
\lim_{n\to\infty}{g_n\left(x-\frac{q(n)}{\sigma_y(n)}\right)}
\ = \ \lim_{n\to\infty}{g_n\left(x+\frac{q(n)}{\sigma_y(n)}\right)}
\ = \ g(x),
\end{align}
it follows that $\{f_n\}_n$ also converges pointwise to $g(x)$.
\end{proof}

\section{Conclusion and Future Work}

We were able to handle the behavior of the number of Zeckendorf summands of numbers  drawn from small intervals by finding a correspondence between Zeckendorf decompositions in the interval $[m, m + F_{\alpha(n)})$ and in the interval $[0, F_{\alpha(n)})$ when a certain gentle condition is placed on the integers $m$ we consider. The key step was to show that almost surely an integer $m$ chosen uniformly at random from $[F_n, F_{n+1})$ will permit the construction of a bijection onto the interval $[0, F_{\alpha(n)})$. Our results follow from previous results on the Gaussian behavior of the number of Zeckendorf summands in this interval.

Our arguments hold for more general recurrence relations (see \cite{BDEMMTTW}), though the arguments become more technical. There are two approaches to proving an analogue of the key step, specifically showing that for almost all $m$ we have a sufficiently large gap in the middle section. One approach is to appeal to some high powered machinery that shows the distribution of the longest gap between summands for $m \in [F_, F_{n+1})$ is strongly concentrated about $C \log\log n$, where $C$ is some constant depending on the recurrence. Results along these lines are known for many recurrences; see \cite{B-AM,BILMT}. Of course, these results contain far more than we need; we do not need to know there is a gap as large as $C \log \log n$, but rather just that there is a gap a little longer than the length of the recurrence.


\ \\

\end{document}